\RequirePackage{fix-cm}
\documentclass[smallextended]{svjour3}
\usepackage{amssymb}
\usepackage{amsfonts}
\usepackage{amsmath}
\usepackage{pstricks}
\usepackage{graphicx}

\smartqed

\renewcommand\bigskip\medskip

\def\F{\mathbb{F}}
\def\N{\mathbb{N}}
\def\Z{\mathbb{Z}}

\def\Q{\mathbb{Q}}
\def\leq{\leqslant}
\def\geq{\geqslant}

\begin{document}

\title{On a two-valued sequence  and related continued fractions
in power series fields }
\author{Bill Allombert, Nicolas Brisebarre and Alain Lasjaunias}

\institute{Bill Allombert \at
Institut de Math\'ematiques de Bordeaux  CNRS-UMR 5251\\
Universit\'e de Bordeaux \\
Talence 33405, France \\
\email{Bill.Allombert@math.u-bordeaux.fr}\\
\and
  Nicolas Brisebarre \at
CNRS, Laboratoire LIP (CNRS, ENSL, Inria, UCBL)\\
Universit\'e de Lyon\\
ENS Lyon, 46 All\'ee d'Italie\\ 
69364 Lyon Cedex 07, France\\
\email{Nicolas.Brisebarre@ens-lyon.fr}\\
\and
Alain Lasjaunias \at
Institut de Math\'ematiques de Bordeaux  CNRS-UMR 5251\\
Universit\'e de Bordeaux \\
Talence 33405, France \\
\email{Alain.Lasjaunias@math.u-bordeaux.fr}
}




\date{Received: date / Accepted: date}

\maketitle

\begin{abstract}
We explicitly describe a {noteworthy} transcendental continued fraction in the field of power series over $\Q$, having irrationality measure equal to 3. This continued fraction is a generating function of a particular sequence in the set $\lbrace 1,2\rbrace$. The origin of this sequence, whose study was initiated in a recent paper, is to be found in another continued fraction, in the field of power series over $\F_3$, which satisfies a simple algebraic equation of degree 4, introduced thirty years ago by D. Robbins. 
\end{abstract}

\section{Introduction}

{In this paper, we present a remarkable continued fraction, built from  an infinite word containing only the letters 1 and 2. The partial study of the structure of this word shows that this sequence is complex enough to imply surprising arithmetical 
properties for the resulting generating function. This generating function is indeed transcendental and it has a very peculiar continued fraction expansion that we will fully describe. We first start by describing our framework and we recall the origin of this strange sequence.}

Here, $K$ is a field which is either the finite field $\mathbb{F}_{q}$, containing $q$ elements and of characteristic $p$, or the field $\Q$ of the rational numbers. Let $T$ be a formal indeterminate. As usual $K[T]$ and $K(T)$ are, respectively, the ring of polynomials and the field of rational functions in $T$ with coefficients in $K$. We consider the field of power series in $1/T$, with coefficients in $K$, denoted by $\mathbb{F}(K)$ or $K((T^{-1}))$. Hence a non-zero element 
of $\mathbb{F}(K)$ can be written as 
\[
\alpha =\sum_{k\leqslant k_{0}}a_{k}T^{k}\quad \textrm{ where} \quad k_{0}\in \Z,\quad a_{k}\in K \textrm{ and}\quad 
a_{k_{0}}\neq 0.
\]
 An ultrametric absolute value is defined over these fields of power series. For $\alpha$ as above we have $|\alpha| =|T|^{k_0}$ where $|T|$ is a fixed real number greater than 1. Note that $\F(K)$ is the completion of the field $K(T)$ for this absolute value.
\par If $K=\mathbb{F}_{q}$, we simply write $\mathbb{F}(q)$ for $\mathbb{F}(K)$. The case $K=\mathbb{F}_{q}$ is particularly important for several reasons. The first one is the link between these fields of power series and certain sequences taking a finite number of values. The fields $\F(q)$ are analogues of the field of real numbers and, as in the case of real numbers, if $\alpha \in \mathbb{F}(q)$ then the sequence of coefficients (or digits) $(a_{k})_{k\leqslant k_{0}}$ for $\alpha $ is ultimately periodic if and only
if $\alpha $ is rational, that is, if $\alpha$ belongs to $\mathbb{F}_{q}(T)$ (note that this fails if $K=\Q$).
However, and this is a distinguishing aspect of the formal case, this sequence of
digits can also be characterized for all elements in $\mathbb{F}(q)$ which
are algebraic over $\mathbb{F}_{q}(T)$; see Theorem 1. Indeed, a large class of sequences, taking values in a finite set, were introduced around the 1960's by computer scientists. These sequences tend to generalize the particularly simple periodic sequences and are called automatic sequences. A full account on this topic and a
very complete list of references are to be found in the book of Allouche and
Shallit \cite{AS}. Automatic sequences are classified by an integer parameter $k\geq 2$, and consequently we have $k$-automatic sequences for all $k\geq 2$. The link between these sequences and power series over a finite field appears in the following theorem
found in the work of Christol \cite{C} (see also the article of
Christol, Kamae, Mend\`{e}s France, and Rauzy \cite{CKMFR}).
\begin{theorem}[Christol]
Let $\alpha $ in $\mathbb{F}(q)$ with $q=p^s$. Let $(a_{k})_{k\leqslant
k_{0}}$ be the sequence of digits of $\alpha$ and $u(n)=a_{-n}$ for all
integers $n\geqslant 0$. Then $\alpha $ is algebraic over $\mathbb{F}_{q}(T)$
if and only if the sequence $(u(n))_{n\geqslant 0}$
is $p$-automatic.
\end{theorem}

\par In the fields of power series over $K$, as for real numbers, there is a continued fraction algorithm. The integers are replaced by elements of $K[T]$. The reader may consult \cite{L1} for a general account concerning continued fractions and Diophantine approximation in power series fields. We know that any irrational element $\alpha $ in $\mathbb{F}(K)$ can be expanded as an infinite continued fraction where the partial
quotients $a_{n}$ are polynomials in $K[T]$, all of positive
degree, except perhaps for the first one. Traditionally, this expansion is denoted by $\alpha =[a_0,a_{1},a_{2},\ldots ,a_{n},\ldots ]$. As for real numbers, this expansion is fundamental to measure the quality of the rational approximation of formal power series. The irrationality measure of an irrational power series $\alpha\in \F(K)$ is defined by
\[
\nu(\alpha)=-\limsup_{\vert Q \vert \to \infty}\log(\vert \alpha -P/Q \vert)/\log(\vert Q \vert),
\]
where $P,Q\in K[T]$. This irrationality measure is directly related to the growth of the sequence of the degrees of the partial quotients $a_n$ in the continued fraction expansion of $\alpha$. Indeed we have (see, for instance, \cite[p. 214]{L1})
\[
\nu(\alpha)=2+\limsup_{n\geq 1}(\deg(a_{n+1})/\sum_{1\leq i\leq n}\deg(a_i)).
\]
Here again the difference between $K$ finite and $K=\Q$ is striking. We recall that, by an adaptation of a classical theorem of Liouville, due to Mahler, if $\alpha\in \F(K)$ is algebraic over $K(T)$, such that $[K(T,\alpha):K(T)]=n$ then we have $\nu(\alpha)\in[2,n]$. For algebraic real numbers this irrationality measure is known to be equal to 2 by the celebrated Roth theorem \cite{R}. It was proved by Uchiyama \cite{U}, shortly after Roth theorem, that the same is true for fields of power series over $K$, if $K$ has characteristic zero. In the case $K=\mathbb{F}_{q}$, the situation is more complex, and Mahler \cite{M} observed that we may have $\nu(\alpha)=n$ for certain algebraic elements of degree $n$. This observation was the beginning of the study of Diophantine approximation in positive characteristic. At the same time the study of continued fractions for certain algebraic power series over finite fields was developed.

\par The starting point of this note is a particular continued fraction in the field $\F(3)$. It was introduced by Mills and Robbins \cite{MR}, and it is the root of the following quartic equation: $x^4+x^2-Tx+1=0$. This equation has a unique root in $\F(p)$ for all $p\geq 2$. The continued fraction expansion for this root $\alpha$ in $\F(3)$, calculated by computer,  was conjectured in \cite{MR}. Some ten years later this conjecture was proved in \cite{BR} and also, shortly afterwards, with a different method in \cite{L2}. Here we have $\alpha=[0,a_1,a_2,\dots,a_n,\dots]$ where $a_n=\lambda_nT^{u_n}$, $(\lambda_n)_{n\geq 1}$ is a sequence in $\F_3^*$ and $(u_n)_{n\geq 1}$ another sequence in $\N^*$. The knowledge of this last sequence in $\N^*$ implies that here we have $\nu(\alpha)=2$, even though the sequence of the degrees of the partial quotients ($u_n$) is unbounded. This implies a remarkable property of this element (see  \cite[p. 209]{L2} and also \cite[p. 58]{LY}). In this note, we only are interested in the first sequence $(\lambda_n)_{n\geq 1}$ in $\F_3^*$, i.e., in $\lbrace 1,2\rbrace$. This sequence is described in the following theorem.   
\begin{theorem}
Let $(W_{n})_{n\geqslant 0}$ be the sequence of finite words over the
alphabet $\{{1,2\}}$, defined recursively as follows:
\[
W_{0}=\epsilon ,\textrm{ }W_{1}=1,\quad \textrm{ and }\quad 
W_{n}=W_{n-1},2,W_{n-2},2,W_{n-1},\quad \textrm{ for }\quad n\geqslant 2,
\]
where commas indicate concatenation of words and $\epsilon$ denotes the empty word. Let $W=(w(n))_{n\geqslant 1}$ be the infinite sequence beginning with $W_{n}$ for all $n\geq 0$. Then the sequence $W=(w(n))_{n\geqslant 1}$ is substitutive but it is not $k$-automatic for any integer $k\geq 2$. 
\end{theorem}

 This theorem was proved in  \cite[ Theorem 4]{LY}. The reader can consult \cite{LY} for a brief account of automatic and substitutive sequences, and also naturally \cite{AS} for a full exposition on these matters. 
\par We set $\theta=\sum_{n\geq 1}w(n)T^{-n}$. Then $\theta$ can be considered in $\mathbb{F}(K)$, where $K=\F_p$, for any prime number $p\geq 2$ or $K=\Q$, and $\theta$ is transcendental over $K(T)$ in all cases.
In the case $K=\F_p$, this is simply a straightforward consequence of Theorem 1. Indeed if $\theta$ were algebraic over $\F_p(T)$, then the sequence $(w(n))_{n\geqslant 1}$ would be $p$-automatic, in contradiction with Theorem 2. In the case, $K=\Q$, by Uchiyama's adaptation of Roth's theorem, it will simply follow from the fact that the irrationality measure for $\theta$ is equal to 3.  This will be established in the next section, {as a consequence of our main result Theorem 3, in which we explicitly give the degrees of the partial quotients of the continued fraction expansion for $\theta$ when considered in the field $\Q((T^{-1}))$. We will end our text with a conjectural description of this continued fraction expansion.}


\section{A transcendental continued fraction in $\Q((T^{-1}))$}

\par We start from the sequence $W=(w(n))_{n\geqslant 1}$, defined above in Theorem 2, and  we consider the formal power series $\theta=\sum_{n\geq 1}w(n)T^{-n}\in \Q((T^{-1}))$. We shall see that this power series is not rational and we shall describe partially the infinite sequence of the partial quotients in the theorem below. 
\par First we define some notions needed for a deep study of the sequence $(W_{n})_{n\geqslant 0}$ of finite words over the alphabet $\{{1,2\}}$. If $M$ is a finite word, we let $|M|$ denote the length of this word. This notation on words should not be confused with the use of the absolute value for elements in the power series fields. Hence we have  $|W_0|=0$ and $|W_1| =1$. We set $\ell_n=|W_n|$. By the recursive definition of $(W_{n})_{n\geqslant 0}$, for $n\geq 2$, we have $\ell_n=2\ell_{n-1}+\ell_{n-2}+2$. We will also use the following: if $M=m_1,m_2,\dots,m_n$ is a finite word we set $\phi(M)=m_1T^{n-1}+m_2T^{n-2}+\dots+m_n \in \Q[T]$ and also $\Phi(M)=\phi(M)T^{-|M|} \in \Q(T)$. By extension, for an infinite word, with our definitions of $W$ and $\theta$, we can simply write $\theta=\Phi(W)$. We have the following theorem which is our main result.   

\begin{theorem}
  Let $\theta=\sum_{n\geq 1}w(n)T^{-n}\in \Q((T^{-1}))$, then $\theta$ is irrational and we have the infinite continued fraction expansion $\theta=[0,a_1,a_2,\dots,a_n,\dots]$ with $a_n\in \Q[T]$ for $n\geq 1$. Let $(\ell_n)_{n\geq 0}$ be the sequence of integers defined by
\[
\ell_0=0,\quad \ell_1=1\quad \textrm{ and}\quad \ell_{n+1}=2\ell_n+\ell_{n-1}+2 \quad \textrm{ for}\quad n\geq 1.
\]
Indeed, we have $\ell_n=((2+\sqrt{2})(1+\sqrt{2})^n+(2-\sqrt{2})(1-\sqrt{2})^n)/4-1.$ 
 \newline Let $(d_n)_{n\geq 1}$ be the sequence of integers defined by $d_n=\deg(a_n)$ for $n\geq 1$. Then this sequence is described as follows:
$d_1=d_2=d_3=d_4=1$ and, for $n\geq 1$, we have 
\[
d_{4n+1}=(3\ell_n+\ell_{n-1}+1)/2, \qquad d_{4n+3}=(\ell_n+\ell_{n-1}+1)/2
\]
and
\[
d_{4n+2}=1, \qquad d_{4n+4}=1.
\]
\end{theorem}

\begin{corollary}
The irrationality measure of $\theta$ is equal to 3.
\end{corollary}
\begin{proof} We know that the irrationality measure of $\theta$ satisfies
\[
\nu(\theta)=2+\limsup_{n\geq 1}(d_{n+1}/\sum_{1\leq i\leq n}d_i).
\]
We observe that, for $n\geq 1$, we have $d_{4n+1}>\max(d_{4n},d_{4n-1},d_{4n-2},d_{4n-3})$ and consequently, we have $\nu(\theta)=2+\lim_{n\geq 1}(d_{4n+1}/\sum_{1\leq i\leq 4n}d_i)$. We set $t_n= \sum_{1\leq i\leq 4n}d_i$. We have $t_1=4$ and by induction, we obtain easily $t_n=2+d_{4n+1}$ for $n\geq 1$. Hence, we have $\nu(\theta)=2+\lim_{n\geq 1}(d_{4n+1}/(2+d_{4n+1}))=3$.\qed
\end{proof}  
\par The proof of Theorem 3 is obtained using three lemmas. Each of the first two lemmas  gives a sequence of good rational approximations to $\theta$. 

In the sequel, we shall avoid the comma for the concatenation of words.
\begin{lemma}
For $n\geq 1$, we set $S_n=T^{(\ell_n+\ell_{n-1}+3)/2}(T^{\ell_n+1}-1)$. Then, for $n\geq 1$, there exists $R_n\in \Q[T]$ such that
\[
|\theta-R_n/S_n|=|S_n|^{-\omega_n}\quad \textrm{ with } \quad \omega_n=3-4/(3\ell_n+\ell_{n-1}+5).
\]
\end{lemma}
\begin{proof}    For $n\geq 1$, we introduce the following words
\[
U_n=W_n 2 W_{n-1} \quad \textrm{ and }\quad V_n=2 W_n.
\]
Note that $W_{n+1}=U_n V_n$. We consider the infinite word 
\[
X_1=U_n V_n V_n \cdots V_n \cdots=U_n(V_n)^{\infty}.
\]
 Now let us consider the element $X_2\in\Q((T^{-1}))$ such that $\Phi(X_1)=X_2$. Indeed, we can write
\[
X_2=\frac{\phi(U_n)}{T^{|U_n|}}+\frac{\phi(V_n)}{T^{|U_n|+|V_n|}}+\frac{\phi(V_n)}{T^{|U_n|+2|V_n|}}+\cdots.
\]
Consequently $X_2\in \Q(T)$ and we have
\[
X_2=\frac{\phi(U_n)}{T^{|U_n|}}+\frac{\phi(V_n)}{T^{|U_n|+|V_n|}}.\frac{T^{|V_n|}}{T^{|V_n|}-1}=\frac{(T^{|V_n|}-1)\phi(U_n)+\phi(V_n)}{T^{|U_n|}(T^{|V_n|}-1)}.
\]
We set $N_2=(T^{|V_n|}-1)\phi(U_n)+\phi(V_n)$ and $D_2=T^{|U_n|}(T^{|V_n|}-1)$ so that we have $X_2=N_2/D_2$. In order to simplify the rational $X_2$, we need to introduce the following relation on finite words. For two finite words $A$ and $B$ , we write $A\bullet B$ if and only if the letters ending $A$ and $B$ respectively are different. Note that $A\bullet B$ if and only if $T$ does not divide $\phi(A)-\phi(B)$. For $n=1$, we have $U_1=12$, $V_1=21$ and  $U_1\bullet V_1$. For $n=2$ we can write
\begin{align*}
U_2&=W_22W_1=G_22W_1\\  \textrm{and }  V_2&=2W_2=2W_12W_02W_1=H_22W_1 \textrm{ with } G_2\bullet H_2.
\end{align*}
We set $(F_1,G_1,H_1)=(\epsilon,U_1,V_1)$ and $(F_2,G_2,H_2)=(2W_1,W_2,2W_12)$. So we have $U_1=G_1F_1$, $V_1=H_1F_1$, $U_2=G_2F_2$ and $V_2=H_2F_2$. By induction, we can define three finite words $F_n$, $G_n$ and $H_n$ such that, for $n\geq 1$, 
\[
U_n=G_nF_n \quad V_n=H_nF_n  \quad \textrm{ with }\quad G_n\bullet H_n.
\]
 The sequence $(F_n,G_n,H_n)$ satisfies, for $n\geq 2$, the following formulas:
\[F_n=2W_12W_2\cdots 2W_{n-1},  \quad H_n=2G_{n-1}\quad \textrm{ and }\quad G_n=U_{n-1}H_{n-1}.\]
Now we will compute the length of the words. We have $|U_n|=\ell_n+\ell_{n-1}+1$ and $|V_n|=\ell_n+1$. But we also have $|F_1|=0$ and, for $n\geq 2$, by using a simple induction, we get
\[
|F_n|=n-1+\ell_1+\ell_2+\dots+\ell_{n-1}=(\ell_n+\ell_{n-1}-1)/2.
\]
Consequently we have, for $n\geq 1$, $|U_n|-|F_n|=|G_n|=(\ell_n+\ell_{n-1}+3)/2$.
\newline It is easily verified that for two finite words $A$ and $B$ we have $\phi(AB)=T^{|B|}\phi(A)+\phi(B)$. Therefore, for three finite words $A$, $B$ and $C$ we also have $\phi(AC)-\phi(BC)=T^{|C|}(\phi(A)-\phi(B))$. Now we define $R_n\in \Q[T]$ by $R_n=\phi(G_{n+1})-\phi(G_n)$. Hence, we can write
\[
N_2=T^{|V_n|}\phi(U_n)+\phi(V_n)-\phi(U_n)=\phi(U_nV_n)-\phi(U_n)
\]
and
\[
N_2=\phi(U_nH_nF_n)-\phi(G_nF_n)=T^{|F_n|}(\phi(U_nH_n)-\phi(G_n))=T^{|F_n|}R_n.
\]
Since $|V_n|=\ell_n+1$, we can also write, in agreement with the formula given in the lemma for $S_n$, 
\[
D_2=T^{|F_n|}(T^{|U_n|-|F_n|}(T^{|V_n|}-1))=T^{|F_n|}S_n.
\]
Consequently we have $X_2=R_n/S_n$. Note that we have $N_2(1)=\phi(V_n)(1)\neq 0$ and consequently $R_n(1)\neq 0$. 
\newline Now we need to study the approximation of $\theta$ by $X_2$. First, for $n\geq 1$, we observe the following equalities
\begin{multline*}
W_{n+2}=W_{n+1}2W_{n}2W_{n+1}=W_{n}2W_{n-1}2W_{n}2W_{n}2W_{n}2W_{n-1}2W_{n}\\=U_nV_n^3V_{n-1}V_{n}.
\end{multline*}
Hence for all $n\geq 1$, the infinite word $W$ begins with $U_nV_n^3V_{n-1}V_{n}$ whereas we have $X_1=U_nV_n^{\infty}$. With our notation, if $M_1$ and $M_2$ are two different words, then we have $|\Phi(M_1)-\Phi(M_2)|=|T|^{-t}$, where $t$ is the first position where $M_1$ and $M_2$ differ. Therefore we have $|\theta-R_n/S_n|=|\Phi(W)-\Phi(X_1)|=|T|^{-t_n}$, if $t_n$ is the rank of the first letter differing in the words $U_nV_n^3V_{n-1}V_{n}$ and $U_nV_n^{\infty}$. From this we deduce 
\[
|\theta-R_n/S_n|=|T|^{-t_n}=|S_n|^{-t_n/\deg(S_n)}=|S_n|^{-\omega_n}.
\]
To compute this value $t_n$, we introduce another relation between finite words. For two finite words $A$ and $B$ , we write $A\bullet \bullet B$ if and only if the letters beginning $A$ and $B$ respectively are different. We define the following sequence $(J_n)_{n\geq 1}$ of finite words:
\[
J_1=\epsilon \quad \textrm{ and }\quad J_n=W_{n-1}2W_{n-2}2\cdots W_12 \quad \textrm{ for }\quad n\geq 2.
\]
Then, by induction, we show that there exists a sequence $(I_n)_{n\geq 1}$ of finite words
\[
I_0=1,\quad I_1=21,\quad I_2=1221, \dots \quad \textrm{ and }\quad I_n\bullet \bullet I_{n-1}\quad \textrm{ for }\quad n\geq 1
\]
such that 
\[
V_{n-1}V_n=2J_nI_n \quad \textrm{and} \quad V_n=2J_nI_{n-1}\quad \textrm{ for }\quad n\geq 1.
\]
Hence we have 
\[
W=U_nV_n^32J_nI_n\cdots \quad \textrm{and} \quad X_1=U_nV_n^32J_nI_{n-1}\cdots
\]
and this implies $t_n=|U_n|+3|V_n|+|J_n|+2$. We observe that, for $n\geq 1$, we have $|J_n|=|F_n|=(\ell_n+\ell_{n-1}-1)/2$. Consequently we get 
\[
t_n=\ell_n+\ell_{n-1}+1+3(\ell_n+1)+(\ell_n+\ell_{n-1}-1)/2+2=(9\ell_n+3\ell_{n-1}+11)/2.
\]
Since $\deg(S_n)=(\ell_n+\ell_{n-1}+3)/2+\ell_n+1=(3\ell_n+\ell_{n-1}+5)/2$, we obtain 
\[
\omega_n=t_n/\deg(S_n)=3-4/(3\ell_n+\ell_{n-1}+5),
\]
as stated in the lemma.\qed
\end{proof}

\begin{lemma} For $n\geq 1$, we set $S'_n=T^{3\ell_{n}+\ell_{n-1}+4}-1$. Then, for $n\geq 1$, there exists $R'_n\in \Q[T]$ such that
\[
|\theta-R'_n/S'_n|=|S'_n|^{-\omega'_n}\quad \textrm{ with } \quad \omega'_n=2+(\ell_n+\ell_{n-1}+1)/(6\ell_{n}+2\ell_{n-1}+8).
\]
\end{lemma}
\begin{proof}    For $n\geq 1$, we introduce the following word
\[
U'_n=U_{n+1} 2=W_{n+1} 2 W_{n} 2.
\]
We consider the infinite word 
\[
X_3=U'_n U'_n \cdots U'_n \cdots=(U'_n)^{\infty}.
\]
Let us consider now the element $X_4\in\Q((T^{-1}))$ such that $\Phi(X_3)=X_4$. Indeed, we can write
\[
X_4=\frac{\phi(U'_n)}{T^{|U'_n|}}+\frac{\phi(U'_n)}{T^{2|U'_n|}}+\frac{\phi(U'_n)}{T^{3|U'_n|}}+\cdots.
\]
Consequently $X_4\in \Q(T)$ and we have
\[
X_4=\frac{\phi(U'_n)}{T^{|U'_n|}}.\frac{T^{|U'_n|}}{T^{|U'_n|}-1}=\frac{\phi(U'_n)}{T^{|U'_n|}-1}.
\]
We have $|U'_n|=|U_{n+1}|+1=\ell_{n+1}+\ell_n+2=3\ell_{n}+\ell_{n-1}+4$. Hence $S'_n=T^{|U'_n|}-1$ and we set $R'_n=\phi(U'_n)$. Note that $R'_n(0)\neq 0$ and also $R'_n(1)\neq 0$. 
\newline We shall now study the approximation of $\theta$ by $X_4=R'_n/S'_n$. First, for $n\geq 1$, we observe the following equalities
\[
W_{n+3}=W_{n+2}2W_{n+1}2W_{n+2}=W_{n+1}2W_{n}2W_{n+1}2W_{n+1}2W_{n+2},
\]
\[
W_{n+3}=U'_nW_{n+1}2W_{n+1}2W_{n+2}=(U'_n)^2W_{n-1}2W_n2W_{n+2}.
\]
Hence for all $n\geq 1$, the infinite word $W$ begins with $(U'_n)^2W_{n-1}2W_n2W_{n+2}$ whereas we have $X_3=(U'_n)^{\infty}$.  Therefore we have $|\theta-R'_n/S'_n|=|\Phi(W)-\Phi(X_3)|=|T|^{-t'_n}$, if $t'_n$ is the rank of the first letter differing in the words $(U'_n)^2W_{n-1}2W_n2W_{n+2}$ and $(U'_n)^{\infty}$. This will imply
\[
|\theta-R'_n/S'_n|=|S'_n|^{-\omega'_n}\quad \textrm{ where}\quad \omega'_n=t'_n/\deg(S'_n).
\]
To compute $t'_n$, we have to compare the finite words $W_{n-1}2W_n2W_{n+2}$ and $U'_n$. For $n\geq 1$, let us consider the word $J_n$, introduced in the proof of Lemma 1. For $n\geq 1$,  by induction, we prove that there are two finite words $A_n$ and $B_n$ such that we have
\[
W_{n-1}2W_n2W_{n+2}=J_nA_n,\quad U'_n=J_nB_n\quad \textrm{ and}\quad  A_n\bullet \bullet B_{n}.
\]
Hence we have 
\[
W=(U'_n)^2J_nA_n\cdots \quad \textrm{and} \quad X_3=(U'_n)^2J_nB_n\cdots
\]
and this implies $u'_n=2|U'_n|+|J_n|$. We know that $|U'_n|=\deg(S'_n)=3\ell_{n}+\ell_{n-1}+4$. Since, for $n\geq 1$, we have $|J_n|=(\ell_n+\ell_{n-1}-1)/2$, we get 
\[
\omega'_n=t'_n/\deg(S'_n)=2+(\ell_n+\ell_{n-1}-1)/(2\deg(S'_n))\textrm{ for}\quad n\geq 1.
\]
This gives us the desired value for $\omega'_n$ stated in the lemma.\qed
\end{proof}
\par Before proving Theorem 3, we need one last lemma to establish the irreducibility of the rational functions $R_n/S_n$ and $R'_n/S'_n$.
\begin{lemma}
Let $R_n$, $S_n$, $R'_n$ and $S'_n$ be the polynomials defined in the two preceding lemmas. For $n\geq 1$, we have the equality
\[
R_nS'_n-R'_nS_n=(-1)^n(T-1).
\] 
Consequently, for $n\geq 1$, we have  
\[
\gcd(R_n,S_n)=\gcd(R'_n,S'_n)=1.
\] 
\end{lemma}
\begin{proof}    We shall use the finite words $U_n$, $V_n$, $F_n$, $G_n$, $H_n$ and $U'_n$, introduced in the previous lemmas. For $n\geq 1$, let us introduce in $\Q[T]$ the polynomials:
\[
P_n=T^{|F_{n+1}|+1}(T^{|V_{n+1}|}+1) \quad \textrm{ and}\quad Q_n=T^{|G_n|}.
\]
Note that we have $|F_{n+1}|+1=(3\ell_n+\ell_{n-1}+3)/2$, $|V_{n+1}|=\ell_{n+1}+1$ and  $|G_n|=(\ell_n+\ell_{n-1}+3)/2$. Consequently, using the recurrence relation on $\ell_n$ and the definitions of $S_n$, $S'_n$, $S_{n+1}$ and $S'_{n+1}$, a direct and elementary computation shows that, for $n\geq 1$, we have
\begin{align}
  S'_{n+1}& =P_nS_{n+1}+S'_n,\label{eq1}\\
  S_{n+1}&=Q_nS'_n-S_n\label{eq2}.
\end{align}
Now we shall see that, for $n\geq 1$, the very same equalities as above hold, if $S$ is replaced by $R$. Indeed, for $n\geq 1$, we have
\begin{align*}
R_{n+1} &=\phi(G_{n+2})-\phi(G_{n+1})=\phi(U_{n+1}H_{n+1})-\phi(G_{n+1})\\ 
&=T^{|H_{n+1}|}\phi(U_{n+1})+\phi(H_{n+1})-\phi(G_{n+1})\\
&= T^{|H_{n+1}|}\phi(U_{n+1})+\phi(2G_{n})-\phi(G_{n+1})\\
&=T^{|H_{n+1}|}\phi(U_{n+1})+2T^{|G_n|}+\phi(G_{n})-\phi(G_{n+1})\\
&=T^{|G_n|}(T^{|H_{n+1}|-|G_n|}\phi(U_{n+1})+2)-R_n\\
&=T^{|G_n|}(T\phi(U_{n+1})+2)-R_n\\
&=T^{|G_n|}\phi(U'_{n})-R_n=Q_nR'_n-R_n, 
\end{align*}
since $H_{n+1}=2G_n$ implies $|H_{n+1}|-|G_n|=1$. To establish the next formula, we will use the following observation: for $n\geq 2$ we have
\[
U_n=W_n2W_{n-1}=W_{n-1}2W_{n-2}2W_{n-1}2W_{n-1}=U_{n-1}V_{n-1}V_{n-1}.
\]
Then, we also have
\begin{align*}
R'_{n+1}-R'_n &=\phi(U'_{n+2})-\phi(U'_{n+1})=T\phi(U_{n+2})+2-(T\phi(U_{n+1})+2)\\ 
&=T(\phi(U_{n+2})-\phi(U_{n+1}))=T (\phi(U_{n+1}V_{n+1}V_{n+1}) - \phi(U_{n+1}))\\
&=T(T^{2|V_{n+1}|}\phi(U_{n+1})+\phi(V_{n+1}V_{n+1})-\phi(U_{n+1}))\\
&=T(T^{2|V_{n+1}|}\phi(U_{n+1})+T^{|V_{n+1}|}\phi(V_{n+1})+\phi(V_{n+1})-\phi(U_{n+1}))\\
&=T((T^{2|V_{n+1}|}-1)\phi(U_{n+1})+(T^{|V_{n+1}|}+1)\phi(V_{n+1}))\\
&=T(T^{|V_{n+1}|}+1)((T^{|V_{n+1}|}-1)\phi(U_{n+1})+\phi(V_{n+1}))\\
&=T(T^{|V_{n+1}|}+1)(T^{|F_{n+1}|}R_{n+1})=P_nR_{n+1}.
\end{align*}
Consequently, we have obtained
\begin{align}
  R'_{n+1}&=P_nR_{n+1}+R'_n, \label{eq3} \\
  R_{n+1}&=Q_nR'_n-R_n. \label{eq4}
\end{align}
 Combining  \eqref{eq1}, \eqref{eq2}, \eqref{eq3} and \eqref{eq4}, we have
\begin{equation}\label{eq5}
P_n= (S'_{n+1}-S'_n)/S_{n+1}=(R'_{n+1}-R'_n)/R_{n+1}
\end{equation}
and
\begin{equation}\label{eq6}
Q_n= (S_{n+1}+S_n)/S'_n=(R_{n+1}+R_n)/R'_n.
\end{equation}
For $n\geq 1$, we set $\Delta_n=R_nS'_n-S_nR'_n$. From \eqref{eq5}, we get
\begin{equation}\label{eq7}
\Delta_{n+1}=R_{n+1}S'_n-S_{n+1}R'_n
\end{equation}
and, in the same way, we have from \eqref{eq6}
\begin{equation}\label{eq8}
\Delta_n=S_{n+1}R'_n-R_{n+1}S'_n.
\end{equation}
From \eqref{eq7} and \eqref{eq8}, we obtain
\[
\Delta_{n+1}=-\Delta_n \quad \textrm{ and}\quad \Delta_n=(-1)^{n+1}\Delta_1\quad \textrm{for}\quad n\geq 1.
\]
Finally, for $n=1$, we have
\[
R_1=T^3+2T^2+T-1 \quad \textrm{ and}\quad R'_1=T^6+2T^5+2T^4+T^3+2T^2+T+2,
\]
but also
\[
S_1=T^2(T^2-1)\quad \textrm{ and}\quad S'_1=T^7-1.
\]
From this we get directly $\Delta_1=-T+1$ and therefore $\Delta_n=(-1)^{n}(T-1)$, for $n\geq 1$, as desired. The statement on the greatest common divisor follows, since $R_n(1)\neq 0$ and $R'_n(1)\neq 0$.\qed 
\end{proof}
\vskip 1 cm
{\bf{Proof of Theorem 3:}} We first observe that, due to the very good rational approximations for $\theta$ stated in Lemma 2 and Lemma 3, it is clear that $\theta$ is irrational. Let us start by recalling classical results concerning the continued fraction algorithm in the formal case. Let $\alpha$ be an irrational element in $\F(K)$ and $\alpha=[a_0,a_1,\dots,a_n,\dots]$ its continued fraction expansion. Let $(x_n/y_n)_{n\geq 0}$ be the sequence of convergents to $\alpha$. The polynomials $x_n$ and $y_n$ are defined recursively by the same relation $z_n=a_nz_{n-1}+z_{n-2}$, with initial conditions  $(x_0,x_1)=(a_0,a_0a_1+1)$ and $(y_0,y_1)=(1,a_1)$. We have $\gcd(x_n,y_n)=1$ and $\deg(y_n)=\deg(a_1)+\dots+\deg(a_n)$, for $n\geq 1$. Moreover, we have the following equality
\[
|\alpha-x_n/y_n|=|y_n|^{-2}|a_{n+1}|^{-1}=|y_n|^{-2-\deg(a_{n+1})/\deg(y_n)} \quad \textrm{ for } \quad n\geq 0.
\] 
Finally if $x$, $y$ are in $K[T]$, with $y\neq 0$ and $|\alpha-x/y|<|y|^{-2}$ then $x/y$ is a convergent of $\alpha$.
\newline From the approximations given in Lemma 1 and Lemma 2, since we clearly have $\omega_n >2$ and $\omega'_n >2$ for $n\geq 1$, we can conclude that the rational functions $R_n/S_n$ and $R'_n/S'_n$ are convergent to $\theta$. Hence, for $n\geq 1$, there are three integers $N(n)$, $M(n)$ and $L(n)$ such that we have
\[
R_n/S_n=x_{N(n)}/y_{N(n)}, \quad R'_n/S'_n=x_{M(n)}/y_{M(n)}
\]
and
\[
  R_{n+1}/S_{n+1}=x_{L(n)}/y_{L(n)}.
\]
Since the rational functions $R_n/S_n$ and $R'_n/S'_n$ are in their lowest terms, we have $|S_n|=|y_{N(n)}|$ and $|S'_n|=|y_{M(n)}|$. On the other hand if we set $D_n=\deg(S_n)$ and $D'_n=\deg(S'_n)$, we observe the following
\[
D_n=(3\ell_n+\ell_{n-1}+5)/2 < D'_n=3\ell_n+\ell_{n-1}+4 <D_{n+1} \quad \textrm{ for } \quad n\geq 1.
\]
Consequently, for $n\geq 1$, we have $N(n)<M(n)<L(n)$. Since we have $\deg(y_n)=\sum_{1\leq i\leq n}\deg(a_i)$, we obtain 
\[
D'_n=D_n+d_{N(n)+1}+\dots+d_{M(n)},
\]
and
\[
D_{n+1}=D'_n+d_{M(n)+1}+\dots+d_{L(n)}.
\]
Due to the rational approximation of the convergents, for $n\geq 1$, we  also have
\[
\omega_n=2+d_{N(n)+1}/D_n \quad \textrm{ and } \quad \omega'_n=2+d_{M(n)+1}/D'_n.
\]
From these two equalities, with the values for $\omega_n $ and $\omega'_n$ given in the lemmas, for $n\geq 1$, we get
 \[
d_{N(n)+1}=(3\ell_n+\ell_{n-1}+1)/2 \quad \textrm{ and } \quad d_{M(n)+1}=(\ell_n+\ell_{n-1}+1)/2.
\]
At last, a straightforward computation shows that
\[
D'_n-D_n=d_{N(n)+1}+1\quad \textrm{ and } \quad D_{n+1}-D'_n=d_{M(n)+1}+1.
\]
Since $d_i\geq 1$ for $i\geq 1$, comparing with the above formulas, we conclude that $M(n)=N(n)+2$ and $d_{M(n)}=1$, but also $L(n)=M(n)+2$ and $d_{L(n)}=1$. We observe that $R_1/S_1=(T^3+2T^2+T-1)/(T^4-T^2)=[0,a_1,a_2,a_3,a_4]$. Consequently, we have $N(1)=4$ and by induction, $N(n)=4n$, $M(n)=4n+2$ and $L(n)=4n+4$, for $n\geq 1$. So the proof of the theorem is complete.

\vskip 1 cm
\begin{remark} The reader will observe that the choice of the pair $(1,2)$ in the definition of the infinite word W is arbitrary. Clearly, the structure of $W$ only depends on a pair of symbols $(a,b)$ with $a\neq b$. To introduce the generating function associated to $W$, we could have taken an arbitrary pair $(a,b)$ in $\Q^2$ with $a \neq b$. However, deciding to keep close to the original sequence, we take $(a,b)=(1,2)$.  The structure of the infinite world for an arbitrary $(a,b)$ implies the existence of the rational approximations given in the first lemmas and the pseudo-coprimality obtained in the last one. Note that, with $(a,b)=(1,-1)$ for instance, we get $S_1=T^2(T^2-1)$ and $R_1=(T^2-2)(T-1)$, and the property of coprimality for the pair $(R_1,S_1)$ fails. This coprimality is necessary to get the regularity of the sequence of the degrees as stated in Theorem 3, and when it fails this brings a slight perturbation in the continued fraction expansion of the corresponding generating function. Of course, we are well aware of the many possible generalizations and extensions, concerning the word $W$ itself and the continued fraction expansion which is derived from it, in the formal case, in characteristic zero and also in positive characteristic, but even in the real case by specializing the indeterminate. However, the aim of this note is simply to describe the continued fraction presented here below, thinking that it has the accidental beauty of some singular mathematical objects.
\end{remark}
\vskip 1 cm

\par Finally, we state a conjecture giving the precise form, and not only the degree, of the partial quotients in the continued fraction expansion for $\theta$. However, in order to avoid long and sophisticated arguments, we have not tried to prove what is stated below. The first partial quotients are given by expanding the convergent $R_1/S_1$. We have $R_1/S_1=[0,a_1,a_2,a_3,a_4]$ and a direct computation gives

\[
a_1=T-2,\quad a_2=T/2+1/4,\quad a_3=8T/5+76/25,\quad a_4=-125T/48+25/24.
\]

For the partial quotients from the fifth on, we have the following conjecture.

\begin{conjecture}
      Let $(\ell_n)_{n\geq 0}$ be the sequence of integers defined in Theorem 3. For $n\geq 1$, there exists $(\lambda_{1,n},\lambda_{2,n},\lambda_{3,n},\lambda_{4,n})\in \Q^4$ such that
\begin{align*}
a_{4n+1}&=\lambda_{1,n}(T^{(3\ell_n+\ell_{n-1}+3)/2}+T^{(\ell_n+\ell_{n-1}+1)/2}-2)/(T-1),\\
a_{4n+3}&=\lambda_{3,n}(T^{(\ell_n+\ell_{n-1}+3)/2}-1)/(T-1)\\
\textrm{and }\;\;
a_{4n+2}&=\lambda_{2,n}(T-1), \qquad a_{4n+4}=\lambda_{4,n}(T-1).
\end{align*}
Moreover the sequences $(\lambda_{i,n})_{n\geq 1}$ in $\Q$, for $i=1,2,3$ and $4$, are as follows:
\[
\begin{array}{ll}
\lambda_{1,n}=(-1)^{n+1}r_n^2, & \qquad \lambda_{2,n}=(-1)^{n+1}(r_n^2+r_nr_{n+1})^{-1},\\
\lambda_{3,n}=(-1)^{n+1}(r_n+r_{n+1})^2,& \qquad \lambda_{4,n}=(-1)^{n+1}(r_{n+1}^2+r_nr_{n+1})^{-1}
\end{array}
\]
where the sequence $(r_n)_{n\geq 1}$ in $\Q$ is defined by
\[
r_n=4(2\ell_n-\ell_{n-1}+1)/25  \quad \textrm{ for }\quad  n\geq 1.
\]
\end{conjecture}




\end{document}